\documentclass[reqno]{amsart}
\usepackage[T1]{fontenc}
\usepackage{pgf,pgfarrows,pgfnodes,pgfautomata,pgfheaps,pgfshade,hyperref, amssymb}
\usepackage{amssymb}

\usepackage[english]{babel}
\usepackage[capitalize]{cleveref}
\usepackage{mathtools}
\usepackage[colorinlistoftodos]{todonotes}
\usepackage{soul}
\usepackage[utf8]{inputenc} 
\usepackage[T1]{fontenc}
\usepackage[normalem]{ulem}

\usepackage{tikz}
\usetikzlibrary{arrows}

\usepackage{xcolor}

\usepackage{enumerate}
\usepackage{comment}

\hypersetup{
    colorlinks,
    linkcolor={blue!30!black},
    citecolor={green!50!black},
    urlcolor={blue!80!black}
} 

\usepackage{mathrsfs} 
\usepackage{dsfont}

\makeatletter
\def\section{\@startsection{section}{1}%
  \z@{.5\linespacing\@plus.7\linespacing}
{.8\baselineskip}%
  {\normalfont\fontsize{11}{13}\centering\bfseries}%
}

\makeatletter
\def\subsection{\@startsection{subsection}{2}%
  \z@{.4\linespacing\@plus.7\linespacing}
{.6\baselineskip}%
  {\normalfont\centering\bfseries}%
}

\newcommand{\supp}{\operatorname{supp}}

\newtheorem{theorem}{Theorem}[section]
\newtheorem{proposition}[theorem]{Proposition}

\newtheorem{lemma}[theorem]{Lemma}

\newcounter{thmcounter}

\newcounter{introthmcounter}

\newtheorem{corollary}[theorem]{Corollary}
\theoremstyle{definition}

\newtheorem*{definition*}{Definition}
\newtheorem{question}[theorem]{Question}
\newtheorem*{question*}{Question}
\newcounter{proofcount}
\AtBeginEnvironment{proof}{\stepcounter{proofcount}}

\makeatletter                  
\@addtoreset{claim}{proofcount}
\makeatother

\theoremstyle{remark}

\newtheorem{remark}[theorem]{Remark}

\newtheorem*{remark*}{Remark}

\counterwithin*{equation}{section}

\def\R{{\mathbb R}}
\def\Z{{\mathbb Z}}

\def\N{{\mathbb N}}

\def\T{{\mathbb T}}
\newcommand*\diff{\mathop{}\!\mathrm{d}}

\def\cB{{\mathcal B}}

\title[Infinite sumsets of the form $B+B$ in sets with large density]{Infinite unrestricted sumsets of the form $B+B$ in 
sets with large density}

\author{Ioannis Kousek}
\address{Department of mathematics, University of Warwick}
\email{ioannis.kousek@warwick.ac.uk}

\author{Trist\'an Radi\'c}
\address{Department of mathematics, Northwestern University}
\email{tristan.radic@u.northwestern.edu}

\begin{document}
\maketitle

\begin{abstract}

For a set $A \subset \N$ we characterize in terms of its 
density when there exists an infinite set 
$B \subset \N$ and $t \in \{0,1\}$ such that $B+B \subset A-t$, 
where $B+B : =\{b_1+b_2\colon  b_1,b_2 \in B\}$. 
Specifically, when the lower density 
$\underline{\diff}(A) >1/2$ 
or 
the upper density $\overline{\diff}(A)> 3/4$, the existence of 
such 
a set  $B\subset \N$ and $t\in \{0,1\}$ is assured. 
Furthermore, whenever $\underline{\diff}(A) > 3/4$ or $\overline{\diff}(A)>5/6$, we show that the shift $t$ is 
unnecessary and we also provide examples to show that these bounds are 
sharp. Finally, we construct a syndetic three-coloring of the 
natural numbers that does not contain a monochromatic $B+B+t$ 
for any infinite set $B \subset \N$ and number $t \in \N$.
\end{abstract}

\section{Introduction} \label{Intro}

In an effort to extend Hindman's theorem \cite[Theorem 3.1]{Hindman74},
Erd\H{o}s conjectured the following statement (for example in  
\cite{Erdos77,Erdos80,Erdos_Turan_seq_intergers:1936}), which 
was recently resolved in 
\cite{Kra_Moreira_Richter_Robertson_BBt}. 

\begin{theorem} \label{KMRR theorem}  \cite[Theorem 1.2]{Kra_Moreira_Richter_Robertson_BBt}
For any $A\subset \N$ with positive upper (Banach) density there exists some infinite set $B\subset \N$ and a number $t\in \N$ such that 
$$B \oplus B =\{b_1 + b_2\colon  b_1,b_2\in B,\ b_1 \neq b_2\} \subset 
A-t.$$
\end{theorem}

The restriction $b_1 \neq b_2$ in the sumset 
 is necessary and the nature of the problem changes substantially when it is removed. For instance, it is not hard to exhibit a set $A \subset \N$ of full upper Banach 
density and positive lower density that does not contain $B+B+t$, where $B+B=\{b_1+b_2: b_1,b_2 \in B\}$, 
for any infinite set $B\subset \N$ and any $t\in \N$ (see \cite[Example 2.3]{Kra_Moreira_Richter_Robertson_problems}).

On a relevant direction, it turns out that the structure $B+B$ for infinite 
$B\subset \N$ is not partition regular and this was 
first shown by Hindman in \cite{Hindman79}, who produced a 
three-coloring of $\N$ that does not contain monochromatic 
$B+B$ for any infinite set $B \subset \N$. The 
analogous problem for two-colorings of $\N$ is 
still open (see \cite[Problem E2494]
{Owing_problems}).

The focus of this paper is to provide conditions on the density of a set $A\subset \N$, 
which ensure the existence of an infinite set $B\subset \N$ such that $B+B \subset A-t$, for some $t\in \{0,1\}$, where we emphasize that repeated entries are permitted. These conditions turn out to be considerably more 
restrictive than those of \cref{KMRR theorem}. 

For additional results in infinite 
sumsets we refer the reader to \cite{DNGJLLM, 
host2019short,Moreira_Richter_Robertson19, nathanson1980sumsets}. We also mention that \cref{KMRR theorem} was recently generalized to a broad family of countable amenable groups in \cite{charamaras_mountakis2024finding}. 

To state the main results precisely we recall some standard definitions.
A set $A\subset \N$ has upper density given by 
$$\overline{\diff}(A)=\limsup_{N\to \infty} \frac{|A \cap [1,N]|}{N}$$
and lower density given by 
$$\underline{\diff}(A)=\liminf_{N\to \infty} \frac{|A \cap [1,N]|}{N},$$
where $[1,N]=\{1,\ldots,N\}$, for any $N\in \N$. When the limit 
exists we say that $A\subset \N$ has (natural) density $\diff(A)$, given by 
the common value of $\overline{\diff}(A)$ and $\underline{\diff}(A)$.

The known examples of integer sets not containing infinite sumsets of the form $B+B$ were sets of 
upper density less than $2/3$ and lower density less than $1/2$, but arbitrarily close to these values (see 
\cite[Example 3.6]{Kra_Moreira_Richter_Robertson_problems}). 
This led B. Kra, J. Moreira, F. Richter and D. 
Robertson to conjecture that these bounds were in a sense sharp 
\cite[Conjecture 3.7]
{Kra_Moreira_Richter_Robertson_problems}. In \cref{thrm B+B+t} we show that this holds once we allow a 
potential shift $t\geq 0$, which can actually be 
restricted to $t\in \{0,1\}$. We are able to remove this shift, as 
shown in \cref{thrm B+B}, once we increase the lower bounds on the 
densities and in 
\cref{sec even numbers} we show that this increase is 
necessary. 

Our first main result is the following theorem.

\begin{theorem} \label{thrm B+B}
     Let $A \subset \N$. 
    \begin{enumerate}
    \item If $\overline{\diff}(A)>5/6,$ there is an infinite set $B\subset \N$ such that $B+B \subset A$.
    \item \label{point 2 thrm} If $\underline{\diff}(A)>3/4,$ there is an infinite set $B\subset \N$ such that $B+B \subset A$.
\end{enumerate}
\end{theorem}

When we allow for a shift $t \geq 0$, we can weaken the assumptions on the density.

\begin{theorem} \label{thrm B+B+t}
    Let $A \subset \N$. 
    \begin{enumerate}
    \item If $\overline{\diff}(A)>2/3,$ there is an infinite set $B\subset \N$ and $t\in \{0,1\}$ such that $B+B+t \subset A$.
    \item \label{point 2 thrm +t} If $\underline{\diff}(A)>1/2,$ there is an infinite set $B\subset \N$ and $t\in \{0,1\}$ such that $B+B+t \subset A$.
\end{enumerate}
\end{theorem}

To prove Theorems \ref{thrm B+B} and 
\ref{thrm B+B+t} we use the tools developed in  
\cite{Kra_Moreira_Richter_Robertson_infinitesumsets, Kra_Moreira_Richter_Robertson_BBt} and translate the problem of finding infinite configurations to a statement in ergodic theory. As was shown in \cite{Kra_Moreira_Richter_Robertson_BBt} the containment $B \oplus B \subset A$, for some infinite $B\subset \N$, is related to the existence of a particular system $(X,T)$ and specific points $a,x_1,x_2 \in X$ such that for some increasing sequence $(n_k)_{k \in \N}$ we have $\left( T \times T \right) ^{n_k} \left( a, x_1 \right) \to (x_1,x_2)$ (see \cite[Theorem 2.2]{Kra_Moreira_Richter_Robertson_BBt}). In the present work, we also need that $T^{2n_k}a \to x_2$ in order to ensure that $2B : = \{ 2b : b \in B\} \subset A$. From these two facts we can conclude that $B+B \subset A$. Given that we care simultaneously about the convergence under 
iterations of $T$ and $T^2$, we need a modified version of  Furstenberg's correspondence principle, where instead of building an appropriated ergodic invariant measure in the full shift in two symbols, $(\{0,1\}^{\Z},T)$, we build  an appropriated 
$T^2\times T$-invariant probability measure in 
$(\{0,1\}^{\Z} \times \{0,1\}^{\Z}, T^2 \times T)$. This is 
carried out in Lemmata \ref{FCP} and \ref{FCP 2} and the proofs of Theorems \ref{thrm B+B} and \ref{thrm B+B+t} are presented in \cref{ sec 3}.

The phenomena captured by Theorems \ref{thrm B+B} and \cref{thrm B+B+t} are directly related to the distribution of a set $A$ along the 
even numbers or its distribution along the odd numbers. More 
precisely, we show the following corollary of \cref{thrm B+B}. 

\begin{corollary} \label{thrm even and odd}
    Let $A \subset \N$. 
    
    If one of the following holds:
    \begin{enumerate}
        \item \label{pt 1 even odd} $\overline{\diff} (A \cap 2\N) > 1/3$ or
        \item \label{pt 2 even odd} $\underline{\diff} (A \cap 2\N) > 1/4$, 
    \end{enumerate}
then there is an infinite set $B \subset \N$ such that $B + B \subset A$. 
Similarly,  if one of the following holds:

\begin{enumerate} 
    \setcounter{enumi}{2}
        \item \label{pt 3 even odd} $\overline{\diff} (A \cap (2\N+1)) > 1/3$ or
        \item \label{pt 4 even odd} $\underline{\diff} (A \cap (2\N+1)) > 1/4$, 
\end{enumerate}  
then there is an infinite set $B \subset \N$ such that $B + B + 1 \subset A$.

\end{corollary}

We prove \cref{thrm even and odd} in \cref{sec density even numbers}. In \cref{subsec examples} we establish that the 
density bounds listed in Theorems \ref{thrm B+B} and \ref{thrm B+B+t} 
are 
sharp. Specifically, we construct a set 
$A\subset \N$ with $\overline{\diff}(A)=5/6$ which does not 
contain a sumset $B+B$, for any infinite $B\subset \N$. 
Moreover, we construct a set $A\subset \N$ with 
$\overline{\diff}(A)=2/3$, which does not contain any shifted 
infinite sumset $B+B+t$, for any $t\in \N$. To complete the 
picture, we give 
analogous constructions of sets $A,A'\subset \N$ with 
$\underline{\diff}(A)=3/4$ and $\underline{\diff}(A')=1/2$,
that do not contain infinite sumsets of the form $B+B$ and $B+B+t$, respectively. We note that for the latter 
two sets, $A$ and $A'$, their lower densities are in fact 
natural densities.

The aforementioned examples include \emph{thick} 
sets, i.e. sets which contain arbitrarily large discrete 
intervals.
This lead us to the natural question whether for any \emph{syndetic} set $A\subset \N$ there exists an infinite set $B\subset \N$ and some $t\in \N$ such that $B+B+t \subset A$. We recall that a set $A\subset \N$ is called syndetic if its complement is not thick or, equivalently, if there is a number $d \in \N$ such that $\N \subset A \cup (A-1) \cup \ldots \cup (A-d)$, where $d$ is called the gap of $A$.  We answer this question in the negative and consequently provide an alternative proof of Hindman's result 
\cite{Hindman79} on three-colorings. We note that, all colors being syndetic, this is a 
sharpening of Hindman's construction.

\begin{proposition} \label{cor 3 coloring synsetic}
There exists a coloring $\N=C_1 \cup C_2 \cup C_3$ where the sets $C_1,C_2,C_3$ are all syndetic and there is no monochromatic sumset $B+B$ for any infinite set $B\subset \N$.   
\end{proposition}
The coloring of \cref{cor 3 coloring synsetic} 
is constructed in \cref{syndetic sets}. 

\medskip

\textbf{Acknowledgments.} The authors thank their advisors, Joel Moreira and Bryna Kra, respectively, for introducing them to these problems and providing valuable insight and guidance throughout the making of this paper. 
The first author also expresses gratitude to Jan Kus, whose comment during an informal seminar presentation of this work at the University of Warwick illuminated how one of the equalities presented in 
\cref{subsec examples} can be achieved.

\medskip

\section{Ergodic theoretic terminology and dynamical statements} \label{dynamical statements}

The proofs of Theorems \ref{thrm B+B} and \ref{thrm B+B+t} are ergodic theoretic in 
nature and depend on manipulations of the new dynamical methods developed in 
\cite{Kra_Moreira_Richter_Robertson_infinitesumsets} and extended in 
\cite{Kra_Moreira_Richter_Robertson_BBt} in order to find infinite patterns in sets 
of positive density. Indeed, by handling this new technology 
and by 
extrapolating crucial combinatorial information given by the 
densities of the sets we 
work with, we were able to minimize the technical 
details of both proofs.

In order to formulate our dynamical results we recall some basic 
terminology. 
Throughout, a topological system is a pair $(X,T)$, where $X$ is a compact metric 
space and $T\colon X \to X$ a homeomorphism. If, in addition, there is a $T$-invariant Borel 
probability measure $\mu$ on $X$, that is, $\mu(T^{-1}A)=\mu(A)$, for all $A\in \cB$, 
we call the triple $(X,\mu,T)$ a \emph{measure preserving system} or system for short.  We stress that, as in \cite{Kra_Moreira_Richter_Robertson_infinitesumsets}, all our measure preserving systems have an underlying topological stucture. 

The system $(X,\mu,T)$ is ergodic if $T^{-1}A=A$, for some $A\in \cB$, 
implies that $\mu(A) \in \{0,1\}$. We denote the \emph{support} of 
the measure $\mu$, meaning the smallest closed subset of $X$ that has full measure, by $\supp(\mu)$.

A \emph{F\o lner sequence} $\Phi$ in $\N$ is a sequence of (non-empty) finite sets $N \mapsto \Phi_N \subset \N$, $N\in \N$, which is assymptotically invariant under any shift. In other words, 
$$\lim_{N\to \infty} \frac{\left| \Phi_N \cap \left(t+\Phi_N\right) \right|}{|\Phi_N|}=1,$$
for any $t\in \N$. Given a system $(X,\mu,T)$, a point $a\in X$ is \emph{generic} for $\mu$ 
along a F\o lner sequence $\Phi$, written as $a\in \textbf{gen}(\mu,\Phi)$, if 
$$\mu = \lim_{N\to \infty} \frac{1}{|\Phi_N|} \sum_{n\in \Phi_N } \delta_{T^n a},$$
where $\delta_x$ is the Dirac mass at $x\in X$ and the limit is in the weak* topology. A well-known consequence of the pointwise ergodic theorem is that for an ergodic system $(X,\mu,T)$ and any F\o lner sequence, there is a subsequence $\Phi$ such that $\mu$-almost every point $x\in X$ is generic along $\Phi$. 

Our first main dynamical result is the following 
theorem. In 
this section we will show that it implies \cref{thrm B+B}.

\begin{theorem} \label{Dynamical B+B}
Let $(X,\mu,T)$ be an ergodic system, let $a\in \textbf{gen}(\mu,\Phi)$ for some F\o lner sequence $\Phi$ and $E_1,E_2 \subset X$ be open sets satisfying
\begin{equation} \label{**}
2\mu(E_2) + \mu(E_1)>2. 
\end{equation}
Then, there exist $x_1,x_2 \in X$ and a strictly increasing sequence $(n_k)_{k\in \N}$ of integers, such that $(T \times T)^{n_k}(a,x_1) \to (x_1,x_2)$ as $k\to \infty$ and $(x_1,x_2) \in E_1 \times E_2$.  
\end{theorem}

Recalling terminology from 
\cite{Kra_Moreira_Richter_Robertson_BBt}, given a topological
system $(X,T)$, a point $(x_0,x_1,x_2) \in X^3$ 
will be called a (three-term) \emph{Erd\H{o}s progression} if 
there exists a strictly increasing sequence $(n_k)_{k\in \N}$ in $\N$, such 
that $(T \times T)^{n_k}(x_0,x_1) \to (x_1,x_2)$. We are going 
to use this concept for the systems arising from \cref{FCP}
below.
An obvious, yet useful, observation is that \cref{Dynamical B+B} 
can be restated using this 
terminology. Indeed, Theorem \ref{Dynamical B+B} states that 
when \eqref{**} holds, there always exists an 
Erd\H{o}s progression of the form $(a,x_1,x_2) \in \{a\}\times E_1 \times E_2$.

The next result, proven in 
\cite{Kra_Moreira_Richter_Robertson_BBt}, facilitates the 
translation of the existence of Erd\H{o}s progressions in 
topological systems to a combinatorial statement about infinite 
sumsets in sets of return times related to those systems.

\begin{lemma}\cite[Theorem $2.2$]{Kra_Moreira_Richter_Robertson_BBt} \label{EP=>SUMSETS} 
Fix a topological system $(X,T)$ and $U,V \subset X$ open. If there exists an Erd\H{o}s progression $(x_0,x_1,x_2)\in X^3$ with $x_1\in U$ and $x_2 \in V$, then there exists some infinite set $B\subset \{n\in \N\colon T^nx_0 \in U\}$ such that $\{b_1+b_2\colon b_1,b_2\in B,\ b_1<b_2\} \subset \{n\in \N\colon T^nx_0 \in V\}$.    
\end{lemma}

In order to prove that \cref{Dynamical B+B} implies 
\cref{thrm B+B} we will make use of a new variant of Furstenberg's 
correspondence principle, originally introduced in
\cite{Furstenberg77}, tailored to our purposes. For the rest of 
this section $\Sigma$ denotes the space $\{0,1\}^{\Z}$ and is endowed with 
the product topology, so that it becomes a compact metrizable space. 
We also consider the shift transformation $S\colon \Sigma \to \Sigma$ given by $S(x(n))=x(n+1)$, for any $n\in \Z$, $x=(x(n))_{n\in \Z} \in \Sigma$.

 \begin{lemma} \label{FCP}
 Let $A\subset \N$. 
 \begin{enumerate}
     \item There exists an ergodic system $(\Sigma 
 \times \Sigma, \mu, S^2 \times S)$, an open set $E\subset \Sigma$, a point $a\in \Sigma$ and a F\o lner sequence 
 $\Phi$, such that $(a,a)\in \textbf{gen}(\mu,\Phi)$ and 

 \begin{equation*} 2\mu(\Sigma \times E) + \mu(E \times \Sigma) \geq 3 \left(2\cdot \overline{\diff}(A)-1 \right).
 \end{equation*}
     \item There exists an ergodic system $(\Sigma 
 \times \Sigma, \mu, S^2 \times S)$, an open set $E\subset 
 \Sigma$, a point $a\in \Sigma$ and a F\o lner sequence 
 $\Phi$, such that $(a,a)\in \textbf{gen}(\mu,\Phi)$ and 
 \begin{equation*} 2\mu(\Sigma \times E) + \mu(E \times \Sigma) \geq 2\cdot \underline{\diff}(A) + 2\cdot \overline{\diff}(A)-1.
 \end{equation*}
 \end{enumerate}   
In both cases we have that $A=\{n\in \N: S^n a \in E \}$. 
 \end{lemma}

\begin{proof}
Let $A \subset \N$. By definition 
there is a sequence of integers $(N_k)$ with 
$$\overline{\diff}(A)=\lim_{k\to \infty} \frac{\left|A \cap [1,N_k]\right|}{N_k}.$$
 Associate to the set $A$ a point $a\in \Sigma = \{0,1\}^{\Z}$ via

\begin{equation*}
    a(n) = 
    \left\{ 
    \begin{array}{cc}
        1 & \text{if $ n\in A$ } \\
        0 & \text{otherwise}.
    \end{array}
    \right.
\end{equation*}
Define the clopen set $E=\{x\in \Sigma\colon x(0)=1\}$ and observe 
that, by construction, $A=\{n\in \N\colon S^na \in E\}$. Now let $N'_k=\lfloor N_k/2\rfloor$, for all $k\in \N$ and consider the sequence of Borel probabilities $(\mu_k)$ on $\Sigma \times \Sigma$ given by $$\mu_k=\frac{1}{N'_k} \sum_{n=1}^{N'_k} \delta_{(S^2 \times S)^n(a,a)}.$$
We let $\mu'$ be a weak* accumulation point of $(\mu_k)$, and  
then it is well-known that $\mu'$ is an $S^2 \times S$-invariant measure.
It follows by definition that for all $k\in \N$ we have
\begin{equation} \label{ineq: 4}
\mu_k(\Sigma \times E)=\frac{1}{N'_k} \sum_{n=1}^{N'_k} \delta_{S^na}(E) = \frac{\left| A\cap \lfloor N'_k \rfloor \right| }{N'_k}.
\end{equation}
For any $k\in \N$ we also have
\begin{align}
    \mu_k(\Sigma \times E)&=\frac{1}{N'_k} \sum_{n=1}^{N'_k} \delta_{S^na}(E) = \frac{1}{N'_k} \left( \sum_{n=1}^{N_k} \delta_{S^na}(E) - \sum_{n=N'_k+1}^{N_k} \delta_{S^na}(E)  \right) \nonumber \\ 
& \geq \frac{2}{N_k} (  \sum_{n=1}^{N_k} \delta_{S^na}(E)- \frac{1}{2}N_k)= 2 \frac{\left|A \cap  \lfloor N_k \rfloor \right|}{N_k}-1 \label{ineq: 1}
\end{align} 
and similarly 
\begin{align} 
\mu_k(E \times \Sigma )&=\frac{1}{N'_k} \sum_{n=1}^{N'_k} \delta_{S^{2n}a}(E) \geq \frac{1}{N'_k} \left( \sum_{n=2}^{N_k-1} \delta_{S^{n}a}(E) -\sum_{n=1}^{N'_k}\delta_{S^{2n+1}a}(E) \right)  \nonumber \\
&\geq \frac{2}{N_k} \left( \sum_{n=1}^{N_k} \delta_{S^{n}a}(E) - 2- \frac{N_k}{2} \right)=2\frac{\left| A\cap \lfloor N_k \rfloor \right| }{N_k}- \frac{4}{N_k} - 1. \label{ineq: 6}
\end{align}
Taking limits as $k\to \infty$ in \eqref{ineq: 4},  \eqref{ineq: 1} and \eqref{ineq: 6} we have, by the definition of $\mu'$, the fact that $E \subset X$ is clopen and the choice of $(N_k)$ that
\begin{equation} \label{ineq: 4'}
\mu'(\Sigma \times E) \geq \underline{\diff}(A),
\end{equation}
\begin{equation} \label{ineq: 1'}
\mu'(\Sigma \times E) \geq 2 \cdot \overline{\diff}(A)-1,
\end{equation}
\begin{equation} \label{ineq: 6'}
\mu'(E \times \Sigma ) \geq 2\cdot \overline{\diff}(A)-1,
\end{equation}
respectively. Although $\mu'$ is not necessarily ergodic, we 
can use its ergodic 
decomposition to find an $S^2 \times S$-ergodic 
component of it, call it $\mu$. Without loss of generality we may 
assume that 
$\mu$ is supported on the orbit closure of $(a,a)$, 
since this holds for $\mu'$ by construction. Then by a standard argument (see \cite[ Proposition 3.9]{Furstenberg81}) we see there is a F\o lner sequence $\Phi$ in $\N$, such that $(a,a) \in \textbf{gen}(\mu,\Phi)$. 

Combining \eqref{ineq: 4'} and \eqref{ineq: 6'} we have that 
$$2\mu'(\Sigma \times E) + \mu'(E \times \Sigma ) \geq 2\cdot \underline{\diff}(A) + 2\cdot \overline{\diff}(A)-1$$
and so we may choose the ergodic component $\mu$ as above and 
so that 
$$2\mu(\Sigma \times E) + \mu(E \times \Sigma ) \geq 2\cdot \underline{\diff}(A) + 2\cdot \overline{\diff}(A)-1.$$
This gives $(2).$ In a similar fashion, combining
\eqref{ineq: 1'} and \eqref{ineq: 6'} and choosing the ergodic component $\mu$ accordingly, we recover $(1)$. 
\end{proof}

\medskip

We now use the above construction, the information provided 
by the density of our sets and \cref{EP=>SUMSETS} to deduce 
\cref{thrm B+B} from \cref{Dynamical B+B}. 

\begin{proof}[Proof that \cref{Dynamical B+B} implies \cref{thrm B+B}] $(1):$ Let $A \subset \N$ with $\overline{\diff}(A)>5/6$. 
By part $(1)$ of \cref{FCP} there exists an ergodic system $(\Sigma 
\times \Sigma, \mu, S^2 \times S)$, an open set $E\subset 
\Sigma$, a point $a\in \Sigma$ with $A=\{n\in \N: S^n a \in E\}$ and a F\o lner sequence 
$\Phi$, such that $(a,a)\in \textbf{gen}(\mu,\Phi)$ and 
\begin{equation*} 2\mu(\Sigma \times E) + \mu(E \times \Sigma) \geq 3 \left(2\cdot \overline{\diff}(A)-1 \right)>2.
\end{equation*}
We now apply Theorem \ref{Dynamical B+B} for the ergodic system 
$(\Sigma \times \Sigma, \mu, S^2 \times S)$, the generic point $(a,a)$ and the open sets $E_1=E \times \Sigma$, $E_2=\Sigma \times E$ to recover an Erd\H{o}s progression $\left( (a,a),(x_{10},x_{11}) , (x_{20},x_{21}) \right) \in \Sigma^6$, with
$$(x_{10},x_{11}, x_{20},x_{21}) \in E_1 \times E_2=E \times \Sigma \times \Sigma \times E.$$
Setting $U=E_1$, $V=E_2$ and applying Theorem \ref{EP=>SUMSETS}, we obtain an infinite set $B \subset \N$ such that
$$B \subset \{n\in \N\colon (S^2\times S)^n(a,a) \in E \times \Sigma\}=\{n\in \N\colon S^{2n}a \in E\} $$
and 
$$B \oplus B \subset \{n\in \N\colon (S^2\times S)^n(a,a) \in \Sigma \times E\}=\{n\in \N\colon S^{n}a \in E\}=A.$$
From the former it follows that $2B :=\{b+b: b\in B\} \subset A$, because 
$$S^{2n}a \in E \iff a(2n)=1 \iff 2n \in A.$$ 
This fact together with $B \oplus B \subset A$ imply that $B+B \subset A$, by the simple observation that $B+B = 2B \cup \left( B \oplus B \right)$.

\medskip

$(2):$ Let $A\subset \N$ with 
$\underline{\diff}(A)>3/4$. 
By part $(2)$ of \cref{FCP} there is an ergodic system $(\Sigma 
\times \Sigma, \mu, S^2 \times S)$, an open set $E\subset 
\Sigma$, a point $a\in \Sigma$ with $A=\{n\in \N: S^n a \in E\}$ and a F\o lner sequence 
$\Phi$, such that $(a,a)\in \text{gen}(\mu,\Phi)$ and 
\begin{equation*} 2\mu(\Sigma \times E) + \mu(E \times \Sigma) \geq 2\cdot \underline{\diff}(A) + 2\cdot \overline{\diff}(A)-1 \geq 4\cdot \underline{\diff}(A)-1>2.
\end{equation*}
Repeating the argument given at the end of the
previous 
proof, in view 
of Theorem
\ref{Dynamical B+B} and then Theorem \ref{EP=>SUMSETS} we find an infinite set $B\subset \N$ such that $B+B \subset A$.
\end{proof}

\medskip

\begin{remark}\label{d_+d^ remark}
Notice that in our proof that \cref{Dynamical B+B} implies 
$(2)$ of \cref{thrm B+B} we did not really need the condition 
$$4\cdot \underline{\diff}(A)-1>2, $$
but only the weaker condition that 
$$ 2\cdot \underline{\diff}(A) + 2\cdot \overline{\diff}(A)-1 > 2,$$
which can be rewritten as 
$$ \underline{\diff}(A) + \overline{\diff}(A) > 3/2.$$
\end{remark}

In light of the simple observation made in 
\cref{d_+d^ remark}, we see how 
\cref{Dynamical B+B} also implies the following 
result, which directly implies part \eqref{point 2 thrm} of \cref{thrm B+B}, and may be of independent interest. 

\begin{theorem} \label{thrm B+B d_ and d^}
    Let $A \subset \N$. If $\overline{\diff}(A) + \underline{\diff}(A) > 3/2$, then there is an infinite set $B\subset \N$ such that $B+B \subset A$. 
\end{theorem}

\medskip

We now state our second main dynamical result, 
which implies \cref{thrm B+B+t}.

\begin{theorem} \label{Dynamical B+B+t}  Let $(X,\mu,T)$ be an ergodic system, let $a\in \textbf{gen}(\mu,\Phi)$, for some F\o lner sequence $\Phi$ and $E_1,E_2,F_1,F_2 \subset X$ be open sets with $F_2=T^{-1}E_2$ and  
\begin{equation} \label{*}
2\mu(E_2) + \mu(E_1) + \mu (F_1) >2. 
\end{equation} 
Then, there exist $x_1,x_2 \in X$ such that $(a,x_1,x_2)\in X^3$ is an Erd\H{o}s progression with either $(x_1,x_2) \in E_1 \times E_2$ or $(x_1,x_2) \in F_1 \times F_2$.  
\end{theorem}

The next instance of a correspondence principle is the analogue 
of \cref{FCP} which allows for the transition from the dynamics 
in \cref{Dynamical B+B+t} to the combinatorics in Theorem \ref{thrm B+B+t}. 

\begin{lemma}\label{FCP 2}
Let $A\subset \N$. There exists and open set $E \subset \Sigma$ and a point $a \in \Sigma$ such that $A = \{ n \in \N: S^n a \in E \}$. Moreover,

\begin{enumerate}
    \item There exists an ergodic system $(\Sigma 
\times \Sigma, \mu, S^2 \times S)$ and a F\o lner sequence 
$\Phi$, such that $(a,a)\in \text{gen}(\mu,\Phi)$ and 
\begin{equation*} 
2\mu(\Sigma \times E) + \mu(E \times \Sigma) + \mu(S^{-1}E 
\times \Sigma) \geq 2 \left( 3\cdot \overline{\diff}(A)-1 \right).
\end{equation*}
    \item Similarly, there exists an ergodic system $(\Sigma \times \Sigma, \mu, S^2 \times S)$ and a F\o lner sequence 
$\Phi$, such that $(a,a)\in \text{gen}(\mu,\Phi)$ and 
\begin{equation*} 2\mu(\Sigma \times E) + \mu(E \times \Sigma) + \mu(S^{-1}E \times \Sigma) \geq 2\left( \underline{\diff}(A) + \overline{\diff}(A) \right).
\end{equation*}
\end{enumerate}   
\end{lemma}

The proof of the above is almost identical to that 
of \cref{FCP}, so we omit it. We merely point out
that in the context and notation of \cref{FCP}, 
we have that 
\begin{align*} 
\mu_k(E \times \Sigma ) + \mu_k(S^{-1}E \times \Sigma)&=\frac{1}{N'_k} \sum_{n=1}^{N'_k} \delta_{S^{2n}a}(E) + \frac{1}{N'_k} \sum_{n=1}^{N'_k} \delta_{S^{2n+1}a}(E)  \nonumber \\ &
\geq \frac{1}{N'_k} \sum_{n=2}^{N_k} \delta_{S^{n}a}(E) \geq 2\frac{\left| A\cap \lfloor N_k \rfloor \right| }{N_k}- \frac{2}{N_k}, \label{ineq:7}
\end{align*}
for all $k\in \N$. Taking limits as $k\to \infty$ this yields 
$$\mu'(E \times \Sigma ) + \mu'(S^{-1}E \times \Sigma) \geq 2\cdot \overline{\diff}(A),$$
which is now used instead of \eqref{ineq: 6'}.

We also remark that in order 
to prove \cref{thrm B+B+t} from \cref{Dynamical B+B+t} one 
would use an argument similar to that of recovering \cref{thrm B+B} from \cref{Dynamical B+B}. This time, we need \cref{FCP 2} 
instead of \cref{FCP} and the sets $E_1 = E \times \Sigma$, 
$E_2 = \Sigma \times E$, $F_1 = S^{-1} E \times \Sigma$ 
and $F_2 = \Sigma \times S^{-1} E$.  We omit 
the repetitive details.  

Moreover, by an observation analogous to 
\cref{d_+d^ remark}, we have that \cref{Dynamical B+B+t} implies the following result, which in turn implies part
$(2)$ of \cref{thrm B+B+t}.

\begin{theorem} \label{B+B+t d_ and d^}
  Let $A \subset \N$. If $\overline{\diff}(A) + \underline{\diff}(A) > 1$, then there is an infinite set $B\subset \N$ and a number $t\in \{0,1\}$ such that $B+B+t \subset A$.    
\end{theorem}

In \cref{2 colorings} we give an application of 
\cref{B+B+t d_ and d^} to the question of two-colorings \cite[Problem E2494]{Owing_problems} mentioned in the introduction and show that the possible constructions have restricted density 
values. 

\medskip

\section{Proofs of the dynamical statements} \label{ sec 3}

\subsection{Tools from \cite{Kra_Moreira_Richter_Robertson_BBt}}

As was mentioned in the introduction, the proofs of 
Theorems 
\ref{Dynamical B+B} and \ref{Dynamical B+B+t} follow a scheme 
similar to that of 
\cite{Kra_Moreira_Richter_Robertson_BBt}. We will thus need 
some of the tools introduced there as well as some 
general concepts. 

We recall that a measurable map $\pi\colon X \to Y$ between  two systems $(X,\mu,T)$ and $(Y,\nu,S)$ is called a \emph{factor map} if $\pi_{*} \mu = \nu$ and 
\begin{equation} \label{conjugation}
\pi \circ T = S \circ \pi  \quad \mu \text{-almost everywhere}  
\end{equation}
 In turn, $(Y,\nu,S)$ is called a \emph{factor} of $(X, \mu,T)$ and $(X, \mu,T)$ is called an \emph{extension} of $(Y,\nu,S)$. If, in addition, $\pi$ is 
continuous, surjective and \eqref{conjugation} holds everywhere we 
call $\pi$ a \emph{continuous factor map}. We remark that a factor of an ergodic system is also ergodic. Of special interest are factors with the structure of 
group rotations. More precisely, a \emph{group 
rotation} is a system $(Z,m,R)$ where $Z$ is a compact abelian group with its normalized Haar measure $m$ and $R\colon Z \to Z$ is the rotation map given by $R(z)=z+\alpha$, for some $\alpha \in Z$.
Every ergodic system has a maximal group rotation factor, 
called the \emph{Kronecker} factor and, even though the factor 
map from an ergodic system $(X,\mu,T)$ to its Kronecker 
$(Z,m,R)$ is a priori only measurable, for our purposes we may 
assume that it is also a continuous surjection. Otherwise, the following result shows that we can pass through an extension.

\begin{proposition} \label{continuous Kronecker} \cite[Proposition 3.20]{Kra_Moreira_Richter_Robertson_infinitesumsets}
Let $(X,\mu,T)$ be an ergodic system and let $a \in \textbf{gen}(\mu,\Phi)$ for some F\o lner sequence $\Phi$. Then there is an 
ergodic system $(\Tilde{X},\Tilde{\mu},\Tilde{T})$, a F\o lner 
sequence $\Psi$, a point $\Tilde{a} \in \Tilde{X}$ and a continous 
factor $\Tilde{\pi} \colon \Tilde{X} \to X$, such that $\Tilde{\pi}(\Tilde{a})=a$, $\Tilde{a}\in \textbf{gen}(\Tilde{\mu},\Psi)$ and 
$(\Tilde{X},\Tilde{\mu},\Tilde{T})$ has a continuous factor map to its Kronecker.
\end{proposition}

For the rest of this subsection, $(X, \mu,T)$ 
denotes an ergodic system with continuous factor 
map $\pi:X \to Z$ to its Kronecker $(Z,m, R)$. Let 
$z\mapsto \eta_z$ denote the disintegration of 
$\mu$ over the factor map $\pi$ (for details see
\cite[Section 
5.3]{Einsiedler_Ward11}). Then, following  
\cite{Kra_Moreira_Richter_Robertson_infinitesumsets}, for every $ (x_1,x_2) \in X \times X$, we define
the measures
\begin{equation} \label{lambda}
\lambda_{(x_1,x_2)}=\int_Z \eta_{z+\pi(x_1)} \times \eta_{z+\pi(x_2)}\ dm(z).
\end{equation} 
on $X \times X$. As it was shown in \cite[Proposition 3.11]
{Kra_Moreira_Richter_Robertson_infinitesumsets}, $(x_1,x_2) \mapsto \lambda_{(x_1,x_2)}$ is a \emph{continuous ergodic decomposition} of $\mu \times \mu$, that is, a disintegration of $\mu \times \mu $ where the measures $\lambda_{(x_1,x_2)}$ are ergodic for $\left( \mu \times \mu \right)$-almost every $(x_1,x_2) \in X \times X$ and the map $(x_1,x_2) \mapsto \lambda_{(x_1,x_2)}$ is continuous. 
To find Erd\H{o}s progressions with specified 
first coordinate, a predetermined point $a\in X$, the authors in 
\cite{Kra_Moreira_Richter_Robertson_BBt} also introduced a measure $\sigma_a$ on $X \times X$ given by
\begin{equation}\label{sigma}
\sigma_a= \int_Z \eta_z \times \eta_{2z-\pi(a)}\ dm(z) = \int_Z \eta_{\pi(a)+z} \times \eta_{\pi(a)+2z}\ dm(z).
\end{equation}
Let us denote by $\pi_1\colon  X \times X$ the projection 
$(x_1,x_2) \mapsto x_1$ onto the first coordinate, by 
$\pi_2\colon  X\times X \to X$ the 
projection onto the second coordinate and by $\pi_i \sigma_a$ the 
pushforward of $\sigma_a$ under $\pi_i$, $i=1,2$. In addition to the standing assumptions of this subsection, we let $a\in \textbf{gen}(\mu,\Phi)$ for some F\o lner sequence $\Phi$. We collect 
various results from 
\cite[Section 3]{Kra_Moreira_Richter_Robertson_BBt} in the 
following proposition.

\begin{proposition} \label{results from KMRR}
    Let $\lambda_{(x_1,x_2)}$, for $(x_1,x_2) \in X\times X$, and $\sigma_a$ be the measures on $X \times X$ defined in \eqref{lambda} and \eqref{sigma}, respectively. Then
    \begin{enumerate}
        \item \label{pushforward of sigma} $\pi_1\sigma_a=\mu$ and $\pi_2\sigma_a+T\pi_2\sigma_a=2\mu$. 
        \item \label{support and lambda(a,x)}
For $\sigma_a$-almost every $(x_1,x_2)\in X \times X$ we have that $(x_1,x_2) \in \supp(\lambda_{(a,x_1)})$.

\item \label{generic point for lambda(a,x)} 
There exists 
a F\o lner sequence $\Psi$, such that for $\mu$-almost every $x_1\in X$ 
the point $(a,x_1)$ belongs to $\textbf{gen}(\lambda_{(a,x_1)},\Psi)$. 
    \end{enumerate}
\end{proposition} 

Using these facts we can guarantee many Erd\H{o}s 
progressions. 

\begin{proposition}\label{EP full measure}
Let $(X,\mu,T)$ be an ergodic system and assume there is a 
continuous factor map $\pi\colon  X \to Z$ to its Kronecker factor. Let 
$a\in \textbf{gen}(\mu,\Phi)$, for some F\o lner sequence $\Phi$. 
Then for $\sigma_a$-almost every $(x_1,x_2) \in X \times X$, the point 
$(a,x_1,x_2)$ is an Erd\H{o}s progression. 
\end{proposition}

\begin{proof}
Let $\Psi$ be the F\o lner sequence and $G\subset X$ be the full 
$\mu$-measure set of $x\in X$ such that $(a,x_1)$ belongs to $\textbf{gen}(\lambda_{(a,x_1)},\Psi)$, arising from part (\ref{generic point for lambda(a,x)}) of \cref{results from KMRR}. By \cref{results from KMRR}, part (\ref{pushforward of sigma}), we have that $\sigma_a(G\times X)=\mu(G)=1$ and so, for $\sigma_a$-almost every $(x_1,x_2) \in X\times X$ we have that $(a,x_1) \in \textbf{gen}(\lambda_{(a,x_1)},\Psi)$. In view of \cref{results from KMRR}, part (\ref{support and lambda(a,x)}), it follows that for $\sigma_a$-almost every $(x_1,x_2) \in X\times X$, 
\begin{enumerate}
    \item $(a,x_1) \in \textbf{gen}(\lambda_{(a,x_1)},\Psi)$ and 
    \item $(x_1,x_2) \in \supp(\lambda_{(a,x_1)})$.  
\end{enumerate}
Thus, applying \cite[Lemma 2.4]{Kra_Moreira_Richter_Robertson_infinitesumsets}, 
we have that for $\sigma_a$-almost every $(x_1,x_2)\in X\times X$ the point 
$(a,x_1,x_2)$ is indeed an Erd\H{o}s progression.  
\end{proof}

\medskip

\subsection{Proofs of the main results}

We are now in the position to prove the following 
special case of \cref{Dynamical B+B}.

\begin{theorem} \label{continuous Dynamical B+B}
Let $(X,\mu,T)$ be an ergodic system and assume there is a continuous factor map $\pi$ to its Kronecker. Let $a\in \textbf{gen}(\mu,\Phi)$, for some F\o lner sequence $\Phi$ and $E_1,E_2 \subset X$ be open sets with
\begin{equation} \label{ineq:10}
2\mu(E_2) + \mu(E_1)>2. 
\end{equation}
Then, there is an Erd\H{o}s progression $(a,x_1,x_2)$ such that
$(x_1,x_2) \in E_1 \times E_2$.  
\end{theorem}

\begin{proof}
We need to find an Erd\H{o}s progression of the 
form 
$(a,x_1,x_2)\in X^3$  
with $(x_1,x_2) \in E_1 \times E_2$. By Proposition 
\ref{EP full measure} it suffices to verify that 
$\sigma_a(E_1 \times E_2)>0$. 
For this, it is enough to show that
\begin{equation} \label{ineq:11} 
\sigma_a(E_1\times X)+ \sigma_a(X\times E_2)>1,    
\end{equation}
because $$E_1 \times E_2=\left(E_1\times X\right) \cap \left(X\times E_2\right)$$
and $\sigma_a$ is a probability measure. 
To see why \eqref{ineq:11} holds, observe that 
\begin{align*}
    \sigma_a(E_1 \times X)+\sigma_a(X \times E_2)&=\pi_1\sigma_a(E_1)+\pi_2\sigma_a(E_2) & \\
    &= \mu(E_1)+2\mu(E_2)-\pi_2\sigma_a(T^{-1}E_2) & \quad \text{by Prop. \ref{results from KMRR} (\ref{pushforward of sigma})} \\
    &\geq \mu(E_1)+2\mu(E_2)-1>1 & \quad  \text{by \eqref{ineq:10}},
\end{align*}
which concludes the proof.
\end{proof}

\medskip
In light of 
\cref{continuous Kronecker}, it turns out that
\cref{Dynamical B+B} reduces to 
\cref{continuous Dynamical B+B}. 

\begin{proof}[Proof of Theorem \ref{Dynamical B+B}]
Let $(X,\mu,T)$, $a\in X$ and $(\Tilde{X},\Tilde{\mu},\Tilde{T})$, $\Tilde{\pi}$ and  $\Tilde{a} \in \Tilde{X}$ be as in Proposition \ref{continuous Kronecker}. As in \cref{Dynamical B+B}, let $E_1,E_2 \subset X$ be open sets satisfying
\begin{equation*}
2\mu(E_2) + \mu(E_1) >2. 
\end{equation*}
Then, we define $\Tilde{E_1}=\Tilde{\pi}^{-1}(E_1)$ and 
$\Tilde{E_2}=\Tilde{\pi}^{-1}(E_2)$. 
By the definition of a factor map we know that $\Tilde{\pi}_{*} \Tilde{\mu}= \mu$ and so it follows that 
\begin{equation*}
2\Tilde{\mu}(\Tilde{E_2}) + \Tilde{\mu}(\Tilde{E_1}) >2. 
\end{equation*}
By Theorem \ref{continuous Dynamical B+B} we can find an Erd\H{o}s 
progression  $(\Tilde{a},\Tilde{x}_1,\Tilde{x}_2) \in \Tilde{X}^3$, such that $(\Tilde{x}_1,\Tilde{x}_2) \in \Tilde{E_1} \times \Tilde{E_2}$. The continuity of $\Tilde{\pi}$ allows us to conclude that the triple
$(a, x_1,x_2):=(\Tilde{\pi}(\Tilde{a}),\Tilde{\pi}
(\Tilde{x}_1),\Tilde{\pi}(\Tilde{x}_2)) \in X^3$ is also an Erd\H{o}s 
progression in $(X,T)$ and clearly, by definition, $(x_1,x_2) \in E_1 \times E_2$. 
\end{proof}

Likewise, we can deduce \cref{Dynamical B+B+t} from the following theorem.

\begin{theorem} \label{continuous Dynamical B+B+t}
Let $(X,\mu,T)$ be an ergodic system and assume there is a continuous factor map $\pi$ to its Kronecker. Let $a\in \textbf{gen}(\mu,\Phi)$, for some F\o lner sequence $\Phi$ and $E_1,E_2,E_3 \subset X$ be open sets with $F_2=T^{-1}E_2$ and
\begin{equation} \label{ineq:9} 
2\mu(E_2) + \mu(E_1) + \mu (F_1) >2. 
\end{equation}
Then, there is an Erd\H{o}s progression $(a,x_1,x_2)$ and either $(x_1,x_2) \in E_1 \times E_2$ or $(x_1,x_2) \in F_1 \times F_2$.  
\end{theorem}

\begin{proof}
This time we want to find an Erd\H{o}s progression of the form $(a,x_1,x_2) \in X^3$ with either $(x_1,x_2) \in E_1 \times E_2$ or $(x_1,x_2) \in F_1 \times F_2$. Again, in view of \cref{EP full measure} it suffices to show that either $\sigma_a(E_1 \times E_2)>0$ or $\sigma_a(F_1 \times F_2)>0.$ Similarly to the proof of \cref{continuous Dynamical B+B}, the latter follows from $\sigma_a(F_1 \times X)+\sigma_a(X \times F_2) >1$. Else, we have that 
\begin{equation*} 
\sigma_a(F_1\times X)+ \sigma_a(X\times F_2) \leq 1,   
\end{equation*} 
which, using \cref{results from KMRR} (\ref{pushforward of sigma}), we rewrite as 
\begin{equation} \label{ineq:14} 
\pi_2\sigma_a(F_2) \leq 1-\pi_1\sigma_a(F_1)=1-\mu(F_1).    
\end{equation}
In this case we also have that 
\begin{align*}
\sigma_a(E_1 \times X)+\sigma_a(X\times E_2) &= \pi_1\sigma_a(E_1)+\pi_2\sigma_a(E_2) &\\
&=\mu(E_1)+2\mu(E_2)-T\pi_2\sigma_a(E_2) & \quad 
\text{by Prop. \ref{results from KMRR} 
(\ref{pushforward of sigma})} \\ 
&=\mu(E_1)+2\mu(E_2)-\pi_2\sigma_a(F_2) & \quad 
\text{since $T^{-1}E_2=F_2$}.
\end{align*}
Using \eqref{ineq:14} we obtain that 
$$\sigma_a(E_1 \times X)+\sigma_a(X\times E_2) \geq \mu(E_1)+2\mu(E_2)+\mu(F_1)-1.$$
By \eqref{ineq:9} it finally follows that
$$\sigma_a(E_1 \times X)+\sigma_a(X\times E_2)>1$$
and as before we conclude that $\sigma_a(E_1\times E_2)>0.$    
\end{proof}

The proof of \cref{Dynamical B+B+t} is similar to 
the one of \cref{Dynamical B+B}, and we simply 
need to use \cref{continuous Dynamical B+B+t} 
instead of \cref{continuous Dynamical B+B}.

\medskip

\section{The bounds are sharp and the role of even 
numbers} \label{sec even numbers}

In this section we show that the lower bounds in \cref{thrm B+B} and \cref{thrm B+B+t} are sharp with 
illustrative examples. Here the proportion of even and odd numbers in the sets plays a crucial role. That role becomes clear in the constructions and is even more explicit in \cref{thrm even and odd}, which we prove below. 
Finally we provide several generalizations of \cref{thrm even and odd}.

\medskip

\subsection{Sets of large densities not containing infinite sumsets B+B} \label{subsec examples}

The following proposition shows that the bounds 
established in parts $(1)$ of Theorems \ref{thrm B+B} and \ref{thrm B+B+t} are optimal. 

\begin{proposition} \label{large upper density without B+B}
There exist sets $A, A' \subset \N$ with $\overline{\diff}(A)=2/3$ and 
$\overline{\diff}(A')=5/6$ and such 
that there is no infinite $B\subset \N$ and no number $t\in \N$ 
with $B+B+t \subset A$ nor $B+B \subset A'$.    
\end{proposition}

\begin{proof}

Let $A$ be the subset of $\N$ defined by
\begin{equation*}
    A = \N \cap \bigcup_{n \in \N} [4^n, (2-1/n) \cdot 4^n).
\end{equation*}
It is clear by the definition of $A$ that $\overline{\diff}(A) = 
\diff_{\Phi} (A)$, where $(\Phi_N)$ is the F\o lner sequence 
given by $N \mapsto  [1, (2-1/N) 4^N) \cap \N$. By standard 
calculations we see that $\overline{\diff}(A)=2/3.$ 

Assume there is an infinite $B\subset \N$ and some $t\in \N$ with $B+B+t \subset A$. Then, in particular, for any $b'\in B$ fixed 
there is $b\in B$ arbitrarily large such that $\{2b+t,b+b'+t\}\subset A$. Therefore, $2b\in [4^{n+1}, (2-\frac{1}{n+1}) 4^{n+1}))-t$, for some $n \in \N$ and diving by two we have
$$ 2\cdot 4^n+t/2 \leq b + t < 2 \left(2-\frac{1}{n+1}\right) 4^n + t/2. $$
Since $t$ is fixed and $b'$ is negligible with respect to $b$ (and 
hence with respect to $n$), we see that 
$b+b'+t \in [2\cdot 4^n,\ 4^{n+1}) \subset \N \setminus A$, reaching a 
contradiction. To see this, observe that $\displaystyle 4^{n+1}-2 \left(2-\frac{1}{n+1}\right) 4^n = 2 \cdot \frac{4^{n}}{n+1} \xrightarrow{n\to \infty} \infty$.
This completes the first construction. 

Keeping $A$ as defined before we now set
$A'=A \cup \left( 2\N+1 \right)$. In other words, $A'=A \cup \left( \left(\N \setminus A \right) \cap \left( 2\N+1 \right) \right)$. By the definition of $A$ we have that $\N \setminus A$ 
is a union of discrete intervals of lower density $1/3$ and so 
$\underline{\diff}(\left(\N \setminus A \right) \cap  2\N )=1/6$. Since the complement of $A'$ is precisely $\left( \N \setminus A \right) \cap  2\N $, it follows that $\overline{\diff}(A')=5/6$. 

We now show there is no infinite set $B\subset \N$ with $B+B \subset A'$. Indeed, if there were such a set, we could consider 
an infinite subset $B'\subset B$ consisting of integers with the 
same parity and so we would have that $B'+B' \subset A' \cap 2\N \subset A$. 
This is a contradiction to the first part of this proof.
\end{proof}

Analogously to \cref{large upper density without B+B}, the next result shows that the bounds 
established in parts $(2)$ of Theorems \ref{thrm B+B} and \ref{thrm B+B+t} are optimal as well. 

\begin{proposition} \label{large lower density withou B+B }
There exist sets $A, A' \subset \N$ with $\underline{\diff}(A)=1/2$ and $\underline{\diff}(A')=3/4$ and such 
that there is no infinite $B\subset \N$ and no number $t\in \N$ 
with $B+B+t \subset A$ nor $B+B \subset A'$.    
\end{proposition}

\begin{proof}
Let $A_1,A_2 \subset \N$ be defined by  
\begin{equation*}
    A_1 = \N \cap \bigcup_{n \in \N} [4^n, (2-1/n) \cdot 4^n)
\end{equation*}
and 
\begin{equation*}
    A_2 = \N \cap \bigcup_{n \in \N} [(2+1/n) \cdot 4^n, 4^{n+1}).
\end{equation*}
It was shown in \cref{large upper density without B+B} that 
there is no infinite $B\subset \N$ and $t\in \N$ so that 
$B+B+t \subset A_1$. By symmetry, the same argument shows 
that there is no infinite $B\subset \N$ and $t\in \N$ so that 
$B+B+t \subset A_2$ either. We remark that the set 
$$\N \setminus \left( A_1 \cup A_2 \right) = \N \cap \bigcup_{n \in \N} \left[ (2-1/n) \cdot 
4^n, (2+1/n) \cdot 4^{n} \right]$$ 
has zero density. This boils down to the fact that 
$$\lim_{N\to \infty} \sum_{n=1}^N \frac{1}{n \cdot 4^{N-n}}= 0.$$
To see this note that $n \mapsto 1 \big/ \left( n\cdot 4^{N-n} \right)$ is increasing in $[1,N]$ and so 
$$\lim_{\theta \nearrow 1} \lim_{N\to \infty} \sum_{n=1}^{\lfloor \theta N \rfloor} \frac{1}{n \cdot 4^{N-
n}} + \sum_{n>\lfloor \theta N \rfloor}^N \frac{1}{n 
\cdot 4^{N-n}} \leq \lim_{\theta \nearrow 1} \lim_{N\to \infty} \frac{1}{4^{(1-\theta)N}} + 
\frac{(1-\theta)N}{N}=0.$$
We now consider the set $A\subset \N$ defined by 
$$A = \left(A_1 \cap 2\N \right) \cup \left( A_2 \cap (2\N+1) \right).$$
It follows that $\diff(A)=1/2$ and 
so, in particular, $\underline{\diff}(A)=1/2$. Moreover, we 
claim that there is no infinite set $B\subset \N$ and no 
$t\in \N$ such that $B+B+t \subset A$. Indeed, if this were 
the case we could consider an infinite subset of $B$, call it 
$B'$, with elements of the same parity. Then, if $t\in 2\N$ 
we should have that $B'+B'+t \subset A\cap 2\N \subset A_1$ 
and if $t+1\in 2\N$ we should have $B'+B'+t \subset A\cap 
(2\N+1) \subset A_2$, a contradiction. 

For the second part of the construction we keep $A_1$ and $A_2$ 
as they were defined above. Then, we let $A'\subset \N$ be 
defined by 
$$A'= \left(A_1 \setminus 4\N \right) \cup 
\left( A_2 \setminus (4\N+2) \right).$$
As we showed before $\diff(A_1 \cup A_2)=1$ and from this it 
easily follows that $\diff(A')=3/4$, because we are only 
removing from $A_1 \cup A_2$ a set of density $1/4$. Finally, 
we claim there is no infinite set $B\subset \N$ satisfying 
$B+B\subset A'$. Indeed, given an infinite set $B$ we can
consider an infinite subset $B' \subset B$ with all its 
elements equivalent mod $4$. That is, there exists $j\in 
\{0,1,2,3\}$ so that $b'\equiv j \pmod4$, for any $b'\in B'$. 
Note that if $j\in \{0,2\}$, then $B'+B' \subset 4\N$ and if 
$j\in \{1,3\},$ then $B'+B' \subset 4\N+2$. In the former case 
we must have $B'+B' \subset A_2$ and in the latter that 
$B'+B' \subset A_1$, which both contradict the results 
established in the first part. 
\end{proof}

\begin{remark*} \label{2-coloring of full density set}
    In the first part of the previous proof, we could instead have considered the set $ \Tilde{A} \subset \N$ defined as $ \Tilde{A}= \left(A_1 \cap (2\N+1) \right) \cup \left( A_2 \cap 2\N \right)$. Notice that as before, $ \Tilde{A}$ does not contain an infinite sumset of the form $B+B+t$ either. The set $A \cup \Tilde{A}$ has density $1$, so with this construction we find a $2$-coloring of a set with full density that does not have a monochromatic infinite sumset $B+B+t$. 
\end{remark*}

Notice that the sets constructed in 
\cref{large lower density withou B+B } are both sets whose 
natural density exists. Thus, 
considering this and Theorems \ref{thrm B+B} and 
\ref{thrm B+B+t} we see that information on the value of the 
natural density of a set, when it exists, is no more powerful 
than information on the value of the lower density of a set, 
which always exists. This was also to be expected from Theorems
\ref{thrm B+B d_ and d^} and \ref{B+B+t d_ and d^}. For 
example, from \ref{B+B+t d_ and d^} we already knew that if 
there existed a set $A\subset \N$ with 
$\underline{\diff}(A)=1/2$ and such that no sumset $B+B+t$ was 
contained in $A$, it would have to be the case that 
$\overline{\diff}(A)=1/2$ too, hence $A$ would have a natural density $\diff(A)=1/2$.

\medskip

\subsection{Density over the even numbers}  \label{sec density even numbers}

We begin this section with a simple combinatorial 
lemma about properties of the upper and lower density. The proof is elementary and so we omit it. 

\begin{lemma} \label{density lemma} Let $\mathcal{P}(\N)$ denote the power-set of $\N$.
\begin{enumerate} 
    \item \label{upper density lemma} The function $\overline{\diff} \colon  \mathcal{P}(\N) \to [0,1]$ is sub-additive. That is, for any disjoint sets $A, B\subset \N$, we have $\overline{\diff}(A\cup B) \leq \overline{\diff}(A)+\overline{\diff}(B)$. In the special case that the density of $B$ exists we have $\overline{\diff}(A\cup B) = \overline{\diff}(A)+\diff(B)$. 
    \item \label{lower density lemma} The function $\underline{\diff} \colon  \mathcal{P}(\N) \to [0,1]$ is super-additive. That is, for any disjoint sets $A, B\subset \N$, we have $\underline{\diff}(A\cup B) \geq \underline{\diff}(A)+\underline{\diff}(B)$. In the special case that the density of $B$ exists we have $\underline{\diff}(A\cup B) = \underline{\diff}(A)+\diff(B)$.
    \item \label{density lemma shift part} If $A,B\subset \N$ are such that $B=A+t$, for some $t\in \N$ and $A \cap B = \emptyset$, we have 
    $\overline{\diff}(A\cup B) = 2\cdot \overline{\diff}(A)$ and
    $\underline{\diff}(A\cup B) = 2\cdot \underline{\diff}(A)$.
\end{enumerate}    
\end{lemma}

Considering the previous examples, we can notice that in order to remove the shift, the set $A$ has to contain a sufficiently large amount of even numbers. This intuition is confirmed by \cref{thrm even and odd}. In the following we state a seemingly more general result than \cref{thrm even and odd}. 

\begin{proposition}  \label{thrm A cap 2D}
    Let $A,D \subset \N$. 
    
    If one of the following holds:
    
    \begin{enumerate}
        \item \label{pt 1 2D} $\overline{\diff} (A \cap 2D) > 1/3$ or
        \item \label{pt 2 2D} $\underline{\diff} (A \cap 2D) > 1/4$, 
    \end{enumerate}

then there is an infinite set $B \subset D$ such that $B + B \subset A$. 

    Likewise, if one of the following holds:
    \begin{enumerate}
        \setcounter{enumi}{2}
        \item \label{pt 1' 2D} $\overline{\diff} \left(A \cap ( 2D+1 )\right) > 1/3$ or
        \item \label{pt 2' 2D} $\underline{\diff} \left(A \cap (2D+1)\right) > 1/4$, 
    \end{enumerate}
then there is an infinite set $B \subset D$ such that $B + B +1 \subset A$. 
\end{proposition}

\begin{proof}
$(1):$ For $A, D \subset \N$ such that $\overline{\diff} (A \cap 2 D) > 1/3$ we let $\Tilde{A}= (A \cap 2D) \cup (2\N +1) $. Then by part (\ref{upper density lemma}) of \cref{density lemma} we see that $\overline{\diff} (\Tilde{A}) > 5/6$  and so by \cref{thrm B+B} there is an infinite set $B \subset \N$ such that $B+B \subset \Tilde{A}$. Without loss of generality we can assume that the elements on $B$ have the same parity, which means that the set $B+B$ consists solely of even numbers. Therefore, $B+B \subset \Tilde{A} \cap 2 \N = A \cap 2D$. In particular, $B+B \subset A$ and $2B \subset 2D$, which also implies that $B \subset D$.   

$(3):$ For $A, D \subset \N$ such that $\overline{\diff} \left(A \cap (2 D+1) \right) > 1/3$ we can shift to obtain that $\overline{\diff} \left((A-1 ) \cap 2 D \right) > 1/3$. Thus, by $(1)$ there is $B\subset D$ infinite with $B+B \subset A-1$ which implies that $B+B+1 \subset A$. 

$(2):$ For $A, D \subset \N$ such that $\underline{\diff} (A \cap 2 D) > 1/4$ we define $\Tilde{A}= (A \cap 2D) \cup (2\N +1) $. Using part \ref{lower density lemma} of 
\cref{density lemma} we have that $\underline{\diff} 
(\Tilde{A}) > 3/4$, so that by \cref{thrm B+B} there is an 
infinite set $B \subset \N$ such that $B+B \subset \Tilde{A}$, 
and repeating the parity argument in $(1)$ we conclude.

$(4):$ This follows from $(2)$ by shifting, exactly like $(3)$ 
follows from $(1)$.
\end{proof}

A special case of interest of Proposition \ref{thrm A cap 2D} is when $D=A$, because 
it gives a criterion to ensure the existence of an 
infinite set $B$ such that $B \cup (B+B) \subset 
A$. Moreover, as was mentioned before, \cref{thrm even and odd} is a consequence of \cref{thrm A cap 2D}. This is clear for \eqref{pt 1 even odd} and \eqref{pt 2 even odd} in \cref{thrm even and odd} by replacing the set $D$ by $\N$. Now notice that $\overline{\diff} (A \cap 2 \N) > 1/3$ implies the existence of an infinite set $B \subset \N$ such that $B+B \subset A$ if and only if $\overline{\diff} (A' \cap (2 \N+1)) > 1/3$ implies $B+B+1 \subset A'$ (where $A' = A + 1 $). That means \eqref{pt 1 even odd} and \eqref{pt 3 even odd} of \cref{thrm even and odd} are equivalent. The same can be done to prove that \eqref{pt 2 even odd} is equivalent to \eqref{pt 4 even odd}.

We would also like to point out that using the sub-additivity 
of the upper density one can also use the equivalent parts $(1)$ and $(3)$ of \cref{thrm even and odd} to deduce \cref{thrm B+B+t} $(1)$ and \cref{thrm B+B} $(1)$. We only show the first implication as the other one is very similar.  

\begin{proof}[Proof that \cref{thrm even and odd} implies \cref{thrm B+B+t} $(1)$]
If $A\subset \N$ has $\overline{\diff}(A)>2/3$, then writing $A=(A\cap 2\N) \cup (A\cap 2\N+1)$, we have by \cref{density lemma} (\ref{upper density lemma}) that either $\overline{\diff}(A\cap 2\N)>1/3$ or $\overline{\diff}(A\cap 2\N+1)>1/3$. 
In the former we see by $(1)$ of \cref{thrm even and odd} that there exists $B\subset \N$ infinite and such that $B+B \subset 
A$ and in the latter we apply $(3)$ of the same corollary, to get an infinite $B\subset \N$ such that $B+B+1 \subset A$.
\end{proof}

We finish this section by showing the special role that the 
powers of $2$ plays in the existence of the set $B$. The next 
proposition is a generalization of \cref{thrm even and odd}, 
but they share a similar proof.

\begin{proposition} \label{prop 2n}
    Let $A \subset \N$, $n \geq 1$ and $0 \leq \ell < 2^{n-1} $. 
    
    If one of the following holds:
    \begin{enumerate}
        \item \label{point 1} $\overline{\diff} (A \cap (2^n \N+2 \ell)) > 1/(3 \cdot 2^{n-1})$ or
        \item \label{point 2} $\underline{\diff} (A \cap (2^n\N +2\ell)) > 1/2^{n+1}$, 
    \end{enumerate}
then there is an infinite set $B \subset \N$ such that $B + B \subset A$. 
Similarly,  if one of the following holds:

\begin{enumerate}
    \setcounter{enumi}{2}
        \item \label{point 3}$\overline{\diff} (A \cap (2^n \N+2 \ell + 1)) > 1/(3 \cdot 2^{n-1})$ or
        \item \label{point 4} $\underline{\diff} (A \cap (2^n\N+2\ell + 1)) > 1/2^{n+1}$,
    \end{enumerate} 
then there is an  infinite set $B \subset \N$ such that $B + B + 1 \subset A$.
\end{proposition}

\begin{proof}
    It is clear that (\ref{point 1}) implies (\ref{point 3}), because if $\overline{\diff} (A \cap (2^n \N+2 \ell + 1)) > 1/(3 \cdot 2^{n-1})$ then $\overline{\diff} ( (A-1) \cap (2^n \N+2 \ell)) > 1/(3 \cdot 2^{n-1})$ and then by (\ref{point 1}) there is an infinite set $B \subset \N$ such that $B+B \subset A-1$. A similar observation allows us to conclude that (\ref{point 2}) implies (\ref{point 4}).

    It is also clear that it is enough to prove (\ref{point 1}) for the case $\ell = 0$. Indeed, if $\overline{\diff} (A \cap (2^n \N+2 \ell)) > 1/(3 \cdot 2^{n-1})$, then also $\overline{\diff} ((A - 2 \ell) \cap 2^n \N) > 1/(3 \cdot 2^{n-1})$ and so there is an infinite set $B'$ such that $B'+B' \subset A - 2 \ell$, therefore taking $B=B'+ \ell$ we conclude that $B+B \subset A$.

   We will thus prove by induction in $n$ that whenever $\overline{\diff} (A \cap 2^n \N) >  1/(3 \cdot 2^{n-1})$, there exists $B\subset \N$ infinite so that $B+B \subset A$. Notice that the base case is given by \cref{thrm even and odd}.

    Assuming the result for $n\in \N$, we let $A \subset \N$ satisfy $\overline{\diff} (A \cap 2^{n+1} \N) >  1/(3 \cdot 2^{n})$. We define $\Tilde{A} = (A \cap 2^{n+1} \N) \cup \left( \left(A \cap 2^{n+1} \N \right)+ 2^{n}) \right)$. Then, $\Tilde{A}=\Tilde{A} \cap 2^n \N $ and 
    by part (\ref{density lemma shift part}) of \cref{density lemma} we have 
    $$\overline{\diff} (\Tilde{A} \cap 2^{n} \N) =  2 \cdot \overline{\diff} (A \cap 2^{n+1} \N) > 2/(3 \cdot 2^{n}) = 1/(3 \cdot 2^{n-1}).$$ 
    Thus by the inductive hypothesis, there exists an infinite set $B' \subset \N$ such that $B'+B'\subset \Tilde{A}$. By the pigeonhole principle, if $B'_i = \{ b \in B \colon  b \equiv i \pmod{2^{n}} \}$, $i \in \{0,\ldots,2^{n}-1\}$, then one of those $B'_i$ is infinite. Let $B'_{i_0}$ be such a set and denote it by $B''$. As $B''+B'' \subset 2^n \N +2i_0$ and $B'' + B'' \subset \Tilde{A} \subset 2^{n}\N$, it must be the case that either $i_0=0$ or $i_0=2^{n-1}$.
    
    If $i_0=0$ then $B'' + B'' \subset \Tilde{A} \cap 2^{n+1}\N = A \cap 2^{n+1}\N$ and we are done. If $i_0=2^{n-1}$ then $B'' + B'' \subset \Tilde{A} \cap (2^{n+1}\N + 2^n)=\left(A \cap 2^{n+1} \N \right)+2^n$. Therefore, setting $B = B'' - 2^{n-1}$ we deduce that $B + B \subset A \cap 2^{n+1} \N $. 

    The proof of (\ref{point 2}) is completely analogous.
\end{proof}

\medskip

\section{B+B in syndetic sets and partition regularity} \label{syndetic sets}

\cref{thrm B+B}, \cref{thrm B+B+t} and the 
examples constructed in Section \ref{sec even numbers} offer a characterization of sets
$A\subset \N$ containing $B+B$ or $B+B+1$ for some 
infinite $B\subset \N$ in terms of information 
provided by their upper and lower density. It is 
then natural to ask for a similar characterization 
in terms of upper or lower Banach density. 
We recall that the upper and lower Banach density 
of a set $A\subset \N$, denoted by $\diff^*(A)$ 
and $\diff_*(A)$, respectively, are defined by

\begin{align*}
    \diff^*(A)&=\limsup_{N-M\to \infty} \frac{|A\cap [M,N]|}{N-M} \\
    \diff_*(A)&=\liminf_{N-M\to \infty} \frac{|A\cap [M,N]|}{N-M}.
\end{align*}
It turns out that no conclusions can be drawn from information on the upper Banach density, because the first example given in \cref{subsec examples} has full upper Banach density and does not contain an infinite sumset $B+B+t \subset A$, for any shift $t\in \N$. 

So far, these results do not offer any insight on the relation between the size of a set $A\subset \N$ captured by its lower Banach density and whether it contains $B+B+t$ for infinite $B\subset \N$ and a number $t\geq 0$, apart from the obvious ones following from the results on lower density (for example, since for $A\subset \N$ one always has $\underline{\diff}(A) \geq \diff_{*}(A)$, we know that $\diff_{*}(A)>1/2$ guarantees $B+B+t \subset A$, some $t\in \{0,1\}$).

The question indirectly raised in the last paragraph is also interesting and very natural for another reason: sets with positive lower Banach density are precisely the very well$-$studied syndetic sets. We have the following proposition.

\begin{proposition}\label{syndetic set with no B+B}
There is a syndetic set $A\subset \N$ with gaps of 
length two such that $A$ does not contain any set 
$B+B+t$ for $B \subset \N$ infinite and $t\in \N$. 
\end{proposition}

\begin{proof}
Let $\T$ be the $1$-dimensional torus $\R / \Z$ and for a real number $x$, $\{x\} $ denotes its fractional part, that is $\{x\} = x - \lfloor x \rfloor $. Let $A\subset \N$ be given by

\begin{equation} \label{eq def A}
    A = \left\{ m \in \N \colon  \{ \theta \log_2(m) \} \in U \right\}
\end{equation}
where $\theta >0$ and the sets $U$ and $U+\theta$ (taken mod 1) are 
separated by a positive distance as subsets of the torus 
$\T$. As it was remarked in \cite[Example 3.6]{Kra_Moreira_Richter_Robertson_problems}, 
$A$ does not contain any $B+B+t$ for any infinite set $B \subset \N$ and shift $t\geq 0$. Indeed, for any infinite set $B \subset \N$ and  $t\geq 0$, we can take $b' \in  B$ arbitrary and
then choose $b \in B$ sufficiently large so that the difference

\begin{equation*}
    \left\lvert \log_2 \left( \frac{2b + t}{b+b'+t} \right) -1   \right\lvert 
\end{equation*}
is less than the distance $d(U, U+ \theta)/\theta$. Multiplying by $\theta$ this implies that

\begin{equation*} 
    \left\lvert \theta\log_2 ( 2b + t) - \theta \log_2(b+b'+t) - \theta  \mod 1 \right\lvert  < d(U, U+ \theta). 
\end{equation*}
Thus if $\{\theta \log_2 (b+b' +t) \} \in U$ then $\{ \theta \log_2 (b+b' +t) + \theta \} \in U + \theta$, and by the previous equation, $\{\theta\log_2 ( 2b + t)\} \not \in U$. Therefore, $2b+t$ and $b+b'+t$ cannot be both in $A$ simultaneously.

Consider now $U_0= [0,1/3 + \varepsilon)$, $U_1= [1/3,2/3 + \varepsilon)$ and $U_2= [2/3, 1 + \varepsilon)$ where $0 <\varepsilon < 1/12 $, so that $\{U_1, U_2, U_3 \}$ is a cover of $\T$ with small overlapping. Notice that for $i \in \{0,1,2\}$, the set
\begin{equation*} 
    A_i = \left\{ m \in \N \colon  \{ 1/2 \log_2(m) \} \in U_i \right\}
\end{equation*}
does not contain any $B+B+t$ for any infinite set $B \subset \N$ and shift $t\geq 0$. Now, we define the set $A$ as
\begin{equation*}
    A = (A_0 \cap 3\N) \cup (A_1 \cap (3\N +1) ) \cup (A_2 \cap (3\N +2)).
\end{equation*}
Then $A$ is clearly syndetic, because $A_0 \cup A_1 \cup A_2 = \N$, each $A_i$ is a disjoint union of discrete intervals and each interval in $A_i$ is followed by an interval of $A_{i+1 \pmod3}$,  $i=1,2,3$. Moreover, as the sets $A_i$ overlap it is clear that the gaps of $A$ have size at most $2$, that is,
$A \cup (A-1) \cup (A-2) = \N$ (this is why we need to include a small parameter $\varepsilon>0$ in the definition of the sets $U_i \subset \T$). 

Finally $A$ does not contain any $B+B+t$ for any infinite set $B \subset \N$ and shift $t\geq 0$. Otherwise, without loss of generality, we could assume that the elements of $B$ are 
equal mod $3$, that is, there exists $j \in \{0,1,2\}$ such that for every $b \in B$, $b \equiv j \mod 3$. Then $B+B+t \subset 3\N + i $, where $i \equiv 2j + t \mod 3$. But this would mean that $B+B+t \subset A \cap (3 \N +i) \subset A_i$, which is a contradiction. 
\end{proof}

\medskip

As an immediate corollary of \cref{syndetic set with no B+B} we recover \cref{cor 3 coloring synsetic}. Indeed, using the set $A$ constructed in 
\cref{syndetic set with no B+B} we can build a 
three-coloring of $\N$ with syndetic colors and 
without any monochromatic infinite sumset $B+B$, 
because $\N= A \cup (A-1) \cup (A-2).$

\begin{remark*} \label{remark 3 colors}
    Although \cref{cor 3 coloring synsetic} is a direct consequence of \cref{syndetic set with no B+B}, it can also be proved directly by considering the sets $A_i$, $i=1,2,3$ as in the proof of \cref{syndetic set with no B+B}, but for $U_0= [0,1/3)$, $U_1= [1/3,2/3)$, and $U_2= [2/3, 1)$ and picking colors $C_1,C_2,C_3$ given by

    \begin{align*}
        C_1 &= (A_0 \cap 3\N) \cup (A_1 \cap (3\N +1) ) \cup (A_2 \cap (3\N +2)) \\
        C_2 &= (A_0 \cap (3\N+2)) \cup (A_1 \cap 3\N  ) \cup (A_2 \cap (3\N +1)) \\
        C_3 &= (A_0 \cap (3\N+1)) \cup (A_1 \cap (3\N +2) ) \cup (A_2 \cap 3\N)
    \end{align*}
\end{remark*}

\medskip

Our example in \cref{syndetic set with no B+B} is a set 
$A\subset \N$ with $\diff_{*}(A)=1/3$ that does not contain any 
infinite sumset $B+B+t$ for any $t\in \N$. By the 
trick of 
gluing to $A$ the odd integers which belong to its complement, as in the second part of \cref{large upper density without B+B}, we 
obtain a set $A'=A \cup \left(2\N+1\right)$ with $\diff_{*}
(A')=2/3$ and such that there is no infinite $B\subset \N$ so that 
$B+B \subset A'$. However, we do not yet have a
characterization in terms of information provided by the lower 
Banach density of a set of integers and so we ask the following 
question. 

\begin{question}
Let $A\subset \N$. 
\begin{enumerate}
    \item If $\diff_*(A)>1/3$, does there exist an infinite set $B\subset \N$ and a number $t\in \N$ so that $B+B+t \subset A$?
    \item If $\diff_*(A)>2/3$, does there exist an infinite set $B\subset \N$ so that $B+B\subset A$?
\end{enumerate}
\end{question}

\medskip

\section{Digression on monochromatic B+B in 2-colorings} \label{2 colorings}   

As we mentioned in the introduction, the following question 
of C.J. Owings remains open. 

\begin{question}\cite[Problem E2494]{Owing_problems} \label{Owings}
Is it true that, no matter how $\N$ is partitioned into two sets, 
one of the sets must contain $B+B$ for some infinite $B\subset \N$?
\end{question}

\begin{remark} \label{equivalent coloring formulation}
    By the parity argument repeated in the previous section, this question is equivalent to finding a partition in the even numbers such that one of the sets contains a $B+B$ for some infinite $B \subset \N$. But, by translation, this is  equivalent to the analogous statement for the odd numbers and shifted patterns $B+B+1$. Joining both statements we notice that the question is equivalent to whether any coloring of the natural numbers $\N$ in two colors, must contain a monochromatic $B+B+t$ for some infinite set $B\subset \N$ and a number $t\in \N$. 
\end{remark}

To the best of the authors' knowledge, there has been no progress towards an answer yet. Even \cref{equivalent coloring formulation} seems to be new. Using \cref{B+B+t d_ and d^}, we are able to prove the following result which reduces the problem to a narrower case by adding density restrictions on both sets in the partition. Notice 
that if the answer to Owings' question turns out to be negative, 
then our result sheds some light to how one could potentially find 
a counterexample.

\begin{proposition} \label{2-coloring progress}
Let $\N = C_1 \cup C_2$ be a partition of $\N$. If there is no 
infinite $B\subset \N$, no number $t\in \N$ and $i\in \{1,2\}$ 
such that 
$B+B+t \subset C_i$, then it must be the case that 
\begin{equation} \label{densities a}
\overline{\diff}(C_i)+\underline{\diff}(C_i)=1,    
\end{equation}  
for both $i=1,2$ and so, in particular, 
\begin{equation} \label{densities b}
\overline{\diff}(C_1)=\overline{\diff}(C_2)\ \text{and}\ \underline{\diff}(C_2)=\underline{\diff}(C_1).      
\end{equation}  
\end{proposition}

\begin{proof}
By \cref{B+B+t d_ and d^} we know that if 
$\overline{\diff}(C_i)+\underline{\diff}(C_i)>1$, then there is 
$B\subset \N$ infinite and $t\in \N$ such that $B+B+t \subset C_i$ 
and so, under our assumption, it must be that 
\begin{equation} \label{densities 0}
\overline{\diff}(C_i)+\underline{\diff}(C_i) \leq 1
\end{equation}
for both 
$i=1,2$. As $\{C_1, C_2\}$ forms a partition of $\N$ we have 
that $\N \setminus C_1=C_2$ and $\N \setminus C_2=C_1$, so that 
\begin{equation} \label{densities 1}
\overline{\diff}(C_1)+\underline{\diff}(C_2) = \overline{\diff}(C_2)+\underline{\diff}(C_1) = 1,
\end{equation}
which follows trivially from the definitions of upper and lower 
density. Summing the two leftmost sides in \eqref{densities 1} we 
get
\begin{equation} \label{densities 2}
\overline{\diff}(C_1)+\underline{\diff}(C_1) + \overline{\diff}(C_2) + \underline{\diff}(C_2)= 2   
\end{equation}
and this forces \eqref{densities 0} to take the form 
$$\overline{\diff}(C_i)+\underline{\diff}(C_i) = 1, $$
for both $i=1,2$. This is \eqref{densities a} and along with the 
equalities in \eqref{densities 1} we also obtain \eqref{densities b}.
\end{proof}

\begin{remark*}
Using \cref{thrm B+B+t} one can also place 
further constrictions on the values of the densities of the sets involved in the two-coloring. Indeed, combining \cref{2-coloring progress} with 
\cref{thrm B+B+t}, if the answer of \cref{Owings} is negative, then $\overline{\diff}(C_1)=\overline{\diff}(C_2) 
\in [1/2 , 2/3]$ and $\underline{\diff}(C_1)=\underline{\diff}(C_2) \in [1/3 , 1/2]$. Given \cref{prop 2n}, an analogous observation can be made considering the relative upper and lower densities of the sets $C_1$ and $C_2$ over the set $2^n \N + \ell $ for $n \geq 1$ and $\ell \in \{0, \ldots, 2^n-1\}$, which makes the construction of such a coloring even more restrictive. 
\end{remark*}

\medskip

\bibliographystyle{abbrv} 



\end{document}